\documentclass[11pt]{amsart}
\usepackage{amsmath}
\usepackage{amssymb}
\usepackage{tabularx}
\usepackage{enumerate}
\usepackage{graphicx}
\usepackage{texdraw}

\usepackage{mathrsfs}

\usepackage{amsfonts,amssymb,amsmath}
\usepackage{epsfig}

\makeatletter
\@addtoreset{equation}{section}

\makeatother
\newtheorem{thm}{Theorem}[section]
\newtheorem{lem}[thm]{Lemma}
\newtheorem{cor}[thm]{Corollary}
\newtheorem{prop}[thm]{Proposition}
\newtheorem{remark}[thm]{Remark}
\newtheorem{defn}[thm]{Definition}
\newcommand{\R}{\mathbb{R}}
\newcommand{\ve}{\varepsilon}

\begin{document}
\title[Mean Curvature Flow in a Cone]{Propagation of a Mean Curvature Flow in a Cone$^\S$}
 \thanks{$\S$ This research was partly supported by NSFC (No.11671262).} 
\author[B. Lou] 
{Bendong Lou$^\dag$} 
\thanks{$\dag$ Mathematics and Science College, Shanghai Normal University, Shanghai 200234, China.
{\bf Email:} {\sf lou@shnu.edu.cn (B. Lou)}} 
\date{\today}

\begin{abstract}
We consider a mean curvature flow in a cone, that is, a hypersurface in a cone which moves
toward the opening with normal velocity equaling to the mean curvature, and the contact angle between the
hypersurface and the cone boundary being $\varepsilon$-periodic in its position.
First, by constructing a family of self-similar solutions, we give a priori estimates for the radially
symmetric solutions and prove the global existence. Then we consider the homogenization limit
as $\ve\to 0$, and use {\it the slowest self-similar solution} to characterize the solution, 
with error $O(1)\ve^{1/6}$, in some finite time interval.
\end{abstract}

\maketitle

\baselineskip 18pt
\section{Introduction}
We consider the propagation of a hypersurface in a cone. The law of the motion of the hypersurface is
the following mean curvature flow
\begin{equation}\label{V}
V =  H \quad \mbox{ on } \Gamma_t \subset \Omega,
\end{equation}
where, $\Gamma_t$ denotes a time-dependent hypersurface in a cone $\Omega$, $\Gamma_t$ contacts the
boundary of $\Omega$ with prescribed angles, $V$ and $H$
denote the normal velocity and the mean curvature of $\Gamma_t$, respectively,
and the cone $\Omega$ is defined as
$$
\Omega :=  \left\{ (x,y)\in \mathbb{R}^{N+1}\ \big|\ y\in \R \mbox{ satisfies } y >|x|, \  x=(x_1,...,x_N)\in \R^N \right\}.
$$

Mean curvature flow \eqref{V} as well as its generalized versions have been extensively studied in the last decades.
To name only a few, Gage and Hamilton \cite{GageH}, Grayson \cite{Gr1, Gr2}, Angenent \cite{Ang1, Ang2},
Chou and Zhu \cite{CZ} and references therein considered shrinking closed plane curves driven by \eqref{V}.
Huisken \cite{Hui1, Hui2} etc. considered closed surfaces in higher dimension spaces.
On the other hand, the mean curvature flow (with or without a driving force) in domains with boundaries
were considered by some authors. For example, in case $\Omega$ is a cylinder, it was studied by Altschuler
and Wu \cite{AW1, AW2}, Matano, Lou, et al. \cite{CaiLou, LMN, MNL, YuanLou} etc. Under certain conditions,
it was shown that the flow will converge to a traveling wave (or, translating solution);
In case $\Omega$ is the half space or a sector on the plane, the existence and asymptotic behavior
of the flow were studied in \cite{CGK, ChenGuo, GMSW, GH, Koh, KL, Nao} etc. under the boundary condition: $\Gamma_t$
contacts the boundary of $\Omega$ with constant angles.

In this paper we consider graphic surfaces in the cone $\Omega$, that is, for some function $y= u(x,t)$,
\begin{equation}\label{aaa}
\Gamma_t = \{ (x, u(x,t)) \mid x\in \omega(t)\subset \R^N \} \subset \Omega,
\end{equation}
where $\omega(t)$ is the definition domain of $u$ containing the origin. To avoid sign confusion, the unit normal vector ${\bf n}$ to $\Gamma_t $ will always be
chosen upward, and so
$$
{\bf n} \ = \frac{(-Du,1)}{\sqrt{1+|Du|^2}} \quad \mbox{and} \quad V= \frac{\partial}{\partial t} \Gamma_t \cdot {\bf n}
= \frac{u_t}{\sqrt{1+|Du|^2}}.
$$
The sign of $H $ will be understood in accordance with this choice of the direction of the normal, which means that $H$ is positive at those points where the hypersurface is convex. So,
$$
H= - \mbox{div}_{(x,y)} {\bf n} = \mbox{div}_x  \left[ \frac{Du}{\sqrt{1+|Du|^2}} \right] =\Big( \delta_{ij} - \frac{D_i u D_j u}{1+|Du|^2} \Big) \frac{D_{ij} u}{\sqrt{1+|Du|^2}}.
$$
Thus, the mean curvature flow \eqref{V} is expressed as
\begin{equation}\label{p01}
u_t =  \Big( \delta_{ij} - \frac{D_i u D_j u}{1+|Du|^2} \Big) D_{ij} u , \qquad  x\in \omega(t)\subset \R^N,\ t>0.
\end{equation}
In addition, we require that $\Gamma_t$ contacts $\partial \Omega$, the boundary of $\Omega$, on
the closed curve $\{(x,|x|) \mid x\in \partial \omega(t)\}$ with prescribed angle $\phi\in (0,\frac{\pi}{4})$.
Denoting by $\nu := \frac{1}{\sqrt{2}|x|} (- x, |x|)$ the inner unit normal to
$\partial \Omega$ at point $(x,|x|)$, we obtain the following boundary condition to our problem:
\begin{equation}\label{p02}
{\bf n} \cdot \nu = \frac{ x\cdot Du +|x|}{|x|\sqrt{2(1+|Du|^2)}} = \cos\phi .
\end{equation}
Therefore, our problem can be expressed by the quasilinear parabolic equation \eqref{p01}
with oblique boundary condition \eqref{p02}.

In the special case where $\Gamma_t$ is a radially symmetric surface, we have $u(x,t)=u(r,t)$ with $r=|x|$, and so
$\omega(t) = B_{\xi(t)}(0) := \{x\in \R^N \mid |x|<\xi(t)\}$  for some $\xi(t)$ satisfying
$\xi (t) =  u(\xi (t), t)$. In this case the equation \eqref{p01} is reduced to
\begin{equation}\label{p}
u_t  = \frac{u_{rr} }{1 +u_r^2 } + \frac{(N-1) u_r}{r}, \qquad  0<r<\xi (t),\ t>0,
\end{equation}
and the boundary condition \eqref{p02} becomes
\begin{equation}\label{cond-homo}
u_r (0,t)=0,\quad u_r( \xi (t), t ) = \tan\Big( \frac{\pi}{4} - \phi\Big), \qquad t>0.
\end{equation}
In the rest of the paper we will focus on this symmetric problem. The general
un-symmetric case will be studied later since the gradient
estimate for the solution is far from well understood, as other quasilinear parabolic equations with oblique
boundary conditions.

When the problem is considered in a homogeneous media, the contact angle $\phi$
should be chosen as a constant. Such cases with $N=1$ (that is, $\Omega$ is a two dimensional sector) have
been studied in \cite{CGK,GH,Koh, KL} etc.  However, when the media or the environment
is a heterogeneous one, the contact angle should be non-constant, as considered in \cite{CaiLou, YuanLou} etc.
The boundary condition we will consider in this paper is such one:
\begin{equation}\label{boundary1}
u_r(0,t)=0,\qquad u_r (\xi (t),t ) = k (u(\xi (t),t)) = k (\xi (t)) ,\quad  t>0,
\end{equation}
where, $k$ is a smooth function satisfying
\begin{equation}\label{g-value}
k \mbox{ is } \varepsilon\mbox{-periodic}\quad \mbox{and} \quad  0< k_0 := \min\limits_{u\in [0,\varepsilon]} k(u) \leq  k^0 := \max\limits_{u\in [0,\varepsilon]} k(u) <1.
\end{equation}

Finally, we will impose
\begin{equation}\label{initial}
u(r,0)= u_0 (r) \qquad {\rm for}\ \ r\in [0, \xi(0)]
\end{equation}
as the initial condition to the problem \eqref{p}-\eqref{boundary1}. Here,  $u_0$ is an {\it admissible function},
which means that $u_0 \in C^1$,
$$
u_0(r)>0 \mbox{ in } [0,\xi(0)],\quad u_0(\xi(0))=\xi(0),\quad u'_0(0) =0,\quad u'_0(\xi(0))= k(\xi(0)) \mbox{ and } |u'_0(r)|<1.
$$

\begin{defn}
A function $u(r,t)$ defined for $0\leq r\leq \xi(t),\ 0\leq t <T$ is called a classical solution of \eqref{p}-\eqref{boundary1}-\eqref{initial}
in the time interval $[0,T)$ if
\begin{itemize}
\item[(a)] $u,\ u_r$ are continuous for $0\leq r \leq \xi(t),\ 0\leq t <T$, and $u_{rr},\ u_t$ are continuous for
$0 < r < \xi(t),\ 0 < t <T$;

\item[(b)] $u$ satisfies \eqref{p}-\eqref{boundary1} for $0<r<\xi(t),\ 0<t<T$ and $u(r,0)$ satisfies \eqref{initial}.
\end{itemize}
It is called a time-global classical solution if $T=+\infty$.
\end{defn}

On the well-posedness we have the following result.

\begin{thm}\label{thm:exist}
Assume \eqref{g-value} and that $u_0$ is an admissible function with $u'_0(r)\geq 0$.
Then the problem \eqref{p}-\eqref{boundary1}-\eqref{initial} has a unique time-global classical solution $u(r,t)$.
\end{thm}

The additional condition $u'_0(r) \geq 0$ in this theorem is used only to guarantee the uniform parabolicity of the converted 
equation in a fixed domain (see details in Remark \ref{rem:u''>0}). With this existence result in hand, we next consider the asymptotic behavior of $u$.
The so-called self-similar solutions will play a key role in this field.
When $k\in (0,1)$ is a constant, a self-similar solution of \eqref{p}-\eqref{boundary1} is a solution of the form
$$
u =\sqrt{2pt} \cdot \varphi \Big( \frac{r}{\sqrt{2pt}};  \ k \Big),
$$
where, $p$ is a positive constant and $\varphi = \varphi(z; k)$ solves
\begin{equation}\label{self-eq}
\left\{
 \begin{array}{ll}
\displaystyle  \frac{\varphi'' (z)}{1+ [\varphi'(z)]^2} =  p[\varphi(z) - z \varphi'(z) ]  - \frac{N-1}{z} \varphi'(z),& 0<z< 1,\\
 \varphi'(0)=0,\quad \varphi'(1)= k. &
\end{array}
\right.
\end{equation}
In case $N=1$, such solutions has been constructed in \cite{CGK, GH, Koh} etc., and
they were used to estimate and to characterize the solution of \eqref{p}-\eqref{boundary1} with constant $k$.
In case $N>1$, however, the construction of such solutions turns out to be much more complicated
due to the presence of the term $(N-1)\varphi'(z) /z$ (see details in the next section).
Moreover, in our current problem, it is easily seen from the nonlinear boundary condition \eqref{boundary1}
that $u$ moves with violently changing instantaneous speeds near $r=\xi(t)$. Hence it is very hard
to give a precise estimate for $u(r,t)$ by using a single self-similar solution. However,
we can show that, in the homogenization case (i.e., as $\ve\to 0$), a finer estimate
is possible by using a special self-similar solution. More precisely, denote by
$\sqrt{2Pt}\; \Phi \big( \frac{r}{\sqrt{2Pt}}\big)$ the self-similar solution of \eqref{self-eq} with
$k= k_0 := \min k(u)$, then we have the following estimate.

\begin{thm}\label{thm:est}
Assume, for some $s_0 > t_0 >0$,
\begin{equation}\label{initial-est}
\sqrt{2P t_0}\; \Phi \Big( \frac{r}{\sqrt{2P t_0}} \Big)\leq u_0 (r) \leq
\sqrt{2P s_0}\; \Phi \Big( \frac{r}{\sqrt{2P s_0}}\Big),
\end{equation}
in their common domains, and $u(r,t)$ is the time-global solution of the problem
\eqref{p}-\eqref{boundary1}-\eqref{initial} obtained in the previous theorem.  If $\ve\ll 1$, 
then for $t\in [0, O(1) \ve^{-4/3}]$ there holds
$$
   \sqrt{2P(t+t_0)}\; \Phi \Big( \frac{r}{\sqrt{2P(t+t_0)}} \Big)\leq u(r,t)
   \leq \sqrt{2P(t+s_0)}\; \Phi \Big( \frac{r}{\sqrt{2P(t+s_0)}}\Big) +
   O(1)\ve^{1/6}
$$
in their common domains.
\end{thm}

The important feature in this result is that $u$ is estimated by {\it the slowest self-similar solution} $\sqrt{2Pt} \Phi(\cdot; k_0)$ 
(it is the slowest one since it moves slower than all the other self-similar solutions $\sqrt{2pt} \varphi(\cdot; k)$ with $k\in (k_0, k^0]$), 
rather than other self-similar solutions or some kinds of average of them. The reason for this self-similar solution being selected is, roughly speaking,
the nonlinear effect in the problem is taken only on the boundary. For a solution starting from the slowest self-similar solution,
the boundary condition \eqref{boundary1} accelerates it a little bit near the boundary but can not speed up the
whole solution (especially, the middle part of the solution) essentially. This kind of result is quite
different from the common homogenization problems where the homogenization limits generally depend on the harmonic or arithmetic averages
of the spatial heterogeneity.
Note that the estimate given in this theorem holds only in some finite time interval (it is wide when $\ve\ll 1$). 
Due to the violent oscillation for the derivative of the solution on the boundary, it is still difficult 
 to give a uniform estimate in the whole time interval $[0,\infty)$ (see Remark \ref{rem:best-est}
for details).

This paper is arranged as the following. In subsection 2.1 we construct self-similar solutions with
prescribed constant angles on the boundaries. In subsection 2.2 we use the slowest/fastest self-similar
solutions as lower/upper solutions to give the $L^\infty$ estimate for the solution,  and use the maximum principle to give the gradient estimate.
In subsection 2.3 we convert our problem into a complicated quasilinear one in a fixed domain
and give its global existence result. Based on this result we prove Theorem 1.1 in subsection 2.4.
In section 3 we consider the homogenization limit as $\ve\to 0$, and prove Theorem 1.2
by constructing a series of delicate upper solutions.

\section{Well-posedness}

\subsection{Self-similar  solutions for the problem with prescribed constant angles}
A self-similar solution of \eqref{p01} is a solution of the form $u(x,t) = \sqrt{2pt} \cdot w(\frac{x}{\sqrt{2pt}})$
for some $p>0$, where $w= w(\tilde{x})$ satisfies
\begin{equation}\label{self-w-eq}
p[w -\tilde{x} \cdot Dw ]= \Big( \delta_{ij} - \frac{D_i w D_j w}{1+|Dw|^2} \Big) D_{ij} w.
\end{equation}
In particular, when $u$ is a radially symmetric function (so is $w$), the self-similar solution is
$$
u = \sqrt{2pt} \cdot w\Big(\frac{x}{\sqrt{2pt}} \Big) =  \sqrt{2pt }\cdot \varphi \left( {\frac {r}{\sqrt{2 pt }}}\right)
$$
with $r= |x|$. Such a function is a solution of \eqref{p}-\eqref{boundary1} with $k=const.$ if $\varphi = \varphi (z; k)$ solves
\begin{equation}\label{ODE}
  \left\{
  \begin{array}{ll}
  \displaystyle  {\frac {\varphi''(z) } {1+ [\varphi'(z)]^2 }} = p[\varphi(z) -z \varphi'(z)] - \frac{N-1}{z} \varphi'(z),  & 0 <z <1,\\
   \varphi'(0) = 0,\quad \varphi (1) =1,\quad \varphi'(1) = k . &
  \end{array}
   \right.
\end{equation}
Clearly, with the additional condition $\varphi(1)=1$, the graph of $y= \sqrt{2 pt}\; \varphi \left( {\frac {r}
{\sqrt{2 pt}}}; k\right) $ contacts the line $y=r$
at $r= R(t):= \sqrt{2pt}$.

To avoid the singularity at $z=0$ in the equation, we replace $z$ by $z+\epsilon$ for any given $\epsilon \in (0,1)$
in the last term of the equation, and first consider the following initial value problem:
\begin{equation}\label{IVP}
  \left\{
  \begin{array}{ll}
  \displaystyle  {\frac {\varphi''(z) } {1+ [\varphi'(z)]^2 }} = p[\varphi(z) -z \varphi'(z)] - \frac{N-1}{z+\epsilon} \varphi'(z),  & z<1,\\
   \varphi (1) = 1,\quad \varphi' (1) = k \in (0,1). &
  \end{array}
   \right.
\end{equation}

First, we fix $k\in (0,1)$ and $\epsilon$ and consider the influence of $p$ on the solutions. For each $p\geq 0$,
by the standard theory of ordinary differential equations, this problem has a
unique solution $\varphi (z)$ in a maximal existence interval $(z_\infty , 1]$ with $z_\infty \geq -\epsilon$.
(In order to emphasize the dependence of $\varphi(z)$ on the parameter $p$, sometimes we also write
$\varphi(z)$ as $\varphi(z;p)$).
Moreover, the initial value condition $\varphi'(1)=k>0$ implies that the graph $\Gamma$ of $\varphi(z)$ goes downward as $z$ decreasing from $1$
and enters the region $D:= (0,1)\times (0,1)$. It will remain in this open domain until one of the following cases happens:
\begin{description}
\item[Case A] $\Gamma$ touches the above boundary $[0,1]\times \{1\}$ of $D$ at $(\bar{z}, \varphi(\bar{z}))= (\bar{z},1)$;
\item[Case B] $\Gamma$ touches the left boundary $\{0\}\times (0,1)$  of $D$ at $(0,\varphi(0))$;
\item[Case C] $\Gamma$ touches the bottom boundary $[0,1]\times \{0\}$  of $D$ at $(\bar{z}, \varphi(\bar{z}))= (\bar{z},0)$.
\end{description}
Another way to classify $\varphi(z)$ is to see whether $\varphi(z)$ has some critical points in $[0,1)$:
\begin{description}
\item[Case 1] there exists $z_* \in (0,1)$ such that $\varphi(z_* )\in (0,1)$, $\varphi'(z_* )=0$ and $\varphi'(z)>0$ for $z\in (z_* , 1]$;
\item[Case 2]  Case B happens, $\varphi(0)\in (0,1)$, $\varphi'(0)=0$ and $\varphi'(z)>0$ for $z\in (0, 1]$;
\item[Case 3] Case B happens, $\varphi'(z)>0$ for $z\in [0, 1]$;
\item[Case 4] Case C happens, $\varphi'(z)>0$ for $z\in [\bar{z}, 1]$;
\item[Case 5] Case C happens, there exists $z^* \in [0,1)$ such that $\varphi(z^* ) = \varphi'(z^* )=0$ and $\varphi'(z )>0$ for $z\in (z^*, 1]$.
\end{description}
When $\epsilon =0$,  Case 2  is what we desired.

\begin{lem}\label{lem:small and large p}
\begin{enumerate}
\item[\rm (i).] Case 5 is impossible.
\item[\rm (ii).] If  Case 1 happens, then $\varphi'(z )<0$ for $z\in (z_\infty , z_* )$.
\end{enumerate}
\end{lem}

\begin{proof}
(i). When Case 5 happens, the equation in \eqref{IVP} with initial data $\varphi(z^* ) = \varphi'(z^* )=0$ has a
unique solution $\varphi(z)\equiv 0$. This contradicts the initial conditions in \eqref{IVP}.

(ii). When Case 1 happens, by the equation of $\varphi$ we have $\varphi'' (z_* ) = p \varphi(z_* ) > 0$.
Thus $\varphi'(z)<0$ for $z$ satisfying $0<z_* -z \ll 1$. We prove the conclusion by contradiction. Assume $z^* <z_* $
is another critical point of $\varphi$, and assume it is the largest one of such points in $(z_\infty , z_* )$. Then
$$
0\geq \varphi''(z^*) = p \varphi(z^*) > p\varphi(z_* )>0,
$$
a contradiction. This proves the lemma.
\end{proof}

From this lemma we see that in Case 1, both Case A and Case B are possible, but the critical point of $\varphi$ is
unique in $(z_\infty , 1)$. Denote
\begin{equation}\label{def:Sigma}
 \Sigma_i  := \{p\geq 0\mid \mbox{Case } i \mbox{ happens}\},\quad i = \mbox{1,\ 2,\ 3 or 4}.
\end{equation}
Clearly, these sets are disjoint each other, and their union is $[0,\infty)$. Set
$$
P_1 := \frac{2\arctan k + (k+4)(N-1)}{1-k},\quad P_2 := \frac{(N-1)k}{2}.
$$

\begin{lem}\label{lem:small and large p}
\begin{itemize}
\item[(i)] $\Sigma_1$ is an open set containing $[P_1, \infty)$;
\item[(ii)] $\Sigma_3 \cup \Sigma_4$ is a bounded open set containing $[0,P_2]$.
\end{itemize}
\end{lem}

\begin{proof}
(i) First we show that $[P_1, \infty) \subset \Sigma_1$. Fix a $p\geq P_1$. By continuity we see that the solution of
\eqref{IVP}  satisfies
$$
\varphi(z) >0,\quad \varphi'(z) >0, \quad \varphi'' (z) >0,
$$
for $z$ satisfying $0<1-z\ll 1$, since they are true at $z=1$.

(a). The following case is impossible: there exists $z_1 \in [\frac12,1)$ such that $\varphi(z_1) =0$
and $\varphi(z)>0,\ \varphi'(z)> 0,\ \varphi''(z)>0$ in $(z_1, 1]$. Otherwise, by continuity we have
$\varphi'(z_1) \geq 0,\ \varphi''(z_1)\geq 0$ and by the equation we have
$$
0\leq \frac{\varphi''(z_1)}{1+ [\varphi'(z_1)]^2 } = -pz_1 \varphi'(z_1) - \frac{N-1}{z_1 +\epsilon} \varphi'(z_1) \leq0 .
$$
Hence $\varphi'(z_1)=0$, and so $\varphi \equiv 0 $ is the unique solution, a contradiction.

(b). The following case is impossible: there exists $z_2\in [\frac12,1)$ such that $\varphi'' (z_2) =0$ and $\varphi (z)> 0,\ \varphi'(z)>0, \ \varphi''(z)>0$ in $(z_2, 1]$.  Otherwise,  with  $\psi(z) := \arctan \varphi'(z) $ we have
\begin{equation}\label{eq-psi}
\psi'' = -pz \varphi'' - \frac{N-1}{z+\epsilon} \varphi'' + \frac{N-1}{(z+\epsilon)^2} \varphi' < 4(N-1)  \varphi' \mbox{ in } (z_2, 1].
\end{equation}
For any $z\in [z_2, 1)$, integrating the above inequality over $[z,1]$ we have
\begin{equation}\label{psi'>}
\psi' (z) >  \psi'(1) - 4 (N-1) [\varphi(1)-\varphi(z)] \geq p(1-k) - \frac{N-1}{1+\epsilon} k  - 4(N-1)  \geq 2\arctan k,
\end{equation}
by $p\geq P_1$.  In particular, $\psi'(z_2) >0$ contradicts our assumption $\varphi''(z_2) =\psi'(z_2) = 0$.

(c). The following case is impossible: $\varphi (z)> 0,\ \varphi'(z)>0, \ \varphi''(z)>0$ in $[\frac12 , 1]$.  Otherwise,  by integrating \eqref{psi'>} over $[\frac12,1]$ we have
$$
\arctan k  < \psi(1)-\psi\Big( \frac12 \Big) < \psi(1) =  \arctan k,
$$
a contradiction.

From the above discussion we see that the only possible case is: there exists $z_3\in [\frac12,1)$ such that $\psi(z_3) = \varphi' (z_3) =0$ and $\varphi (z)> 0,\ \varphi'(z)>0, \ \varphi''(z)>0$ in $(z_3, 1]$.  Therefore, any  $p\geq P_1$ belongs to $\Sigma_1$.

For each $p_0\in \Sigma_1$, there exists $z_*(p_0) \in (0, 1)$ such that
$$
\varphi'(z_1; p_0) < 0 = \varphi' (z_* (p_0); p_0) < \varphi'(z_2; p_0) \quad \mbox{  for } z_\infty  < z_1 < z_*(p_0) <z_2 \leq 1.
$$
Fix such a pair $z_1$ and $z_2$, since $\varphi(z;p)$ as well as its derivatives depend on $p$ continuously, we see that $\varphi'(z_1; p)<0 < \varphi'(z_2; p)$ for $p$ satisfying $|p-p_0|\ll 1$. This implies that such $p$ also belongs to $\Sigma_1$, and so $\Sigma_1$ is an open set.

(ii) Taking  $p \in [0, P_2]$ and taking $z=1$ in the equation we have
$$
\frac{\varphi''(1)}{1+ k^2} = p(1-k) -\frac{N-1}{1+\epsilon} k <0,
$$
provided $\epsilon <1 $. By continuity, $\varphi''(z) <0$ for $z$ with $0<1-z\ll 1$.
We claim that $\varphi''(z)<0$ for all $z\in (z_\infty ,1]\cap [0,1]$, and so $p\in \Sigma_3\cup \Sigma_4$.
We prove the claim by contradiction. Assume that $z_4 $ is the rightmost point in
$(z_\infty ,1)\cap [0,1)$ such that $\varphi''(z)=0$. Then
$\varphi''(z)<0$ in $(z_4, 1)$, and so as $z$ decreasing from $1$ to $z_4$, $\varphi'(z)$ becomes larger and larger, while $\varphi(z)$ becomes smaller and smaller. In particular at $z=z_4$ we have, when $\epsilon \in (0,1)$,
$$
0=\frac{\varphi''(z_4)}{1+ [\varphi'(z_4)]^2} = p (\varphi(z_4)- z_4 \varphi'(z_4)) -\frac{N-1}{z_4+\epsilon} \varphi'(z_4) < p  -\frac{N-1}{2} k \leq 0,
$$
a contradiction. This implies that $\Sigma_3\cup \Sigma_4$ contains $[0, P_2]$.

For any $p^0\in \Sigma_3 \cup \Sigma_4$, we have $\min\limits_{\bar{z}\leq z\leq 1} \varphi'(z;p^0) >0$ for $\bar{z}$ in Case 4 or $\bar{z}=0$ in Case 3.
By the continuous dependence we see that the minimum of $\varphi'(z;p)$ is also positive when $p$ is near $p^0$ and when the graph of $\varphi(z;p)$ lies in $\overline{D}$.
This means that $\Sigma_3 \cup \Sigma_4$ is an open set. (Note that $\Sigma_4$ is not necessarily to be open, since it may contain such $\tilde{p}$ that $\varphi(0;\tilde{p})=0$).
\end{proof}

A consequence of the above lemma is the following result.

\begin{cor}\label{cor:Case-2}
$\Sigma_2 = [0,\infty) \backslash (\Sigma_1 \cup \Sigma_3 \cup \Sigma_4) \subset [P_2, P_1]$ is a nonempty and closed set.
\end{cor}

This corollary shows that, for each small $\epsilon>0$, there exists some $p\in [P_2, P_1]$ such that
Case 2 happens for the equation in \eqref{IVP}.   Now we give some a priori (uniform in $\epsilon$) estimates for these solutions, and then take limit as $\epsilon \to 0$ to obtain a solution to \eqref{ODE}.

\begin{lem}\label{lem:varphi''>0}
Assume $\varphi (z;p)$ is a solution of \eqref{IVP}  for some $p\in \Sigma_2$. Then $\varphi''(z;p)>0$ in $[0,1]$.
\end{lem}

\begin{proof}
Taking $z=0$ in the equation we have $\varphi'' (0) = p \varphi (0) >0 $, and so $\varphi' (z)$ is monotonically increasing near $z=0$. If $\varphi'' (z_5) =0$ for some $z_5\in (0,1]$ (without loss of generality, assume $z_5$ is the smallest one of such points in $(0,1]$). Then $\varphi'' (z) >0$ in $(0,z_5)$.

Using  $\psi := \arctan \varphi'$ and the equation in \eqref{eq-psi} at $z=z_5$ we have
$\psi'' (z_5) >0$ since $\varphi'(z_5)>0$. Combining with $\psi' (z_5) = \varphi''(z_5 ) =0$ we see that $\psi$ takes a strict local minimum at $z=z_5$. So is $\varphi ' =\tan \psi$.  This, however, implies that $\varphi'(z)>\varphi'(z_5)$ for $z$ with $0<z_5 -z \ll 1$, contradicting the above conclusion $\varphi''(z) >0$ in $(0,z_5)$.
\end{proof}

By this lemma we have the following a priori estimates.

\begin{lem}\label{lem:est-varphi-epsilon}
Let $\varphi(z)= \varphi(z;p)$ be the solution of \eqref{IVP} with $p\in \Sigma_2$. Then
\begin{itemize}
\item[(i)]  $0< \varphi (z)\leq 1 $ and  $0\leq \varphi' (z) \leq k$ for $z\in [0,1]$;
\item[(ii)] for any $\delta\in (0,1)$ and any integer $m\geq 2$, there exists a positive constant $C= C(\delta, m, N)$  (independent of $\epsilon$) such that
\begin{equation}\label{kth-order-est}
|\varphi^{(m)} (z) | \leq C(\delta,m, N),\quad z\in [\delta, 1].
\end{equation}
\end{itemize}
\end{lem}

\begin{proof}
(i). The conclusions follow from the previous lemma.

(ii). From the previous lemma we see that $p(\varphi - z\varphi')$ is strictly decreasing, and so
$$
P_2 (1-k)\leq p(1-k) \leq p  (\varphi  (z)- z\varphi'  (z))  \leq p\varphi(0)\leq P_1 ,\quad  z\in [0, 1],
$$
since $p\in \Sigma_2 \subset [P_2, P_1]$. In addition,
$$
0<\frac{N-1}{z+\epsilon} \varphi'  (z)  \leq \frac{N-1}{\delta} k,\quad z\in [\delta, 1].
$$
Hence, using the equation of $\varphi$ we then obtain the (uniform in $\epsilon$) estimate for $\varphi'' $
in $[\delta,1]$. Differentiating the equation $(m-2)$-times we can obtain the estimates for $\varphi^{(m)} $ in $[\delta, 1]$.
\end{proof}

Based on the above results we now make $\epsilon$ change and take limit as $\epsilon \to 0$.
For each small $\epsilon >0$, we select one $p\in \Sigma_2 \subset [P_2, P_1]$,  denote it by $p_\epsilon$,
and denote the corresponding solution $\varphi(z; p_\epsilon)$ of \eqref{IVP} simply by $\varphi_\epsilon (z) $, which is the solution in Case 2. Now we consider the limit of $\varphi_\epsilon$ as $\epsilon \to 0$.
Using the estimates in the above lemma and using the Cantor's diagonal argument we can find
a sequence $\{\epsilon_i\}$ decreasing to $0$, a parameter $P \in [P_2, P_1]$ and a function $\Phi\in C([0,1]) \cap C^\infty((0,1])$ such that
\begin{equation}\label{converge-sequence}
p_{\epsilon_i} \to P,\quad \|\varphi_{\epsilon_i} -  \Phi \|_{C([0,1])} \to 0,\quad \|\varphi_{\epsilon_i} -  \Phi \|_{C^m ([\delta,1])} \to 0
\mbox{ as } i\to \infty,
\end{equation}
for any integer $m\geq 1$ and any $\delta \in (0,1)$. Moreover,
\begin{equation}\label{1st-2nd-Phi}
0\leq \Phi(z)\leq 1,\quad 0\leq \Phi'(z) \leq k,\quad \Phi''(z)\geq 0, \quad z\in (0,1]
\end{equation}
by the previous results. Therefore, $\Phi$ satisfies the problem \eqref{ODE} with $p=P$, except for $\varphi'(0)=0$.

Finally we can prove the main result in this subsection.

\begin{prop}\label{prop:exist-self-similar}
For any given $k\in (0,1)$, there exists a unique $P\in [P_2, P_1]$ such that the
problem \eqref{ODE} with $p=P$ has a solution.

Moreover, the unique solution, denoted by $\varphi= \Phi(z;k)  \in C^\infty ([0,1])$ satisfies
\begin{equation}\label{1st-2nd-Phi-strict}
0< \Phi(z;k) <1,\ \  0 <\Phi'(z;k) <k \mbox{ in } (0,1) ,\mbox{\ \ and \ \ } \Phi'' (z;k)>0 \mbox{ in } [0,1].
\end{equation}
\end{prop}

\begin{proof}
Let $P$ and $\Phi (z)\in C([0,1]) \cap C^\infty ((0,1])$ be number and the function obtained in \eqref{converge-sequence}.

(1) First we consider the smoothness of $\Phi$ at $z=0$ by using the equation \eqref{self-w-eq} instead of \eqref{ODE}. Set
$$
W(\tilde{x})  :=  \Phi(z) \mbox{ with } z= |\tilde{x}| \in [0,1],
$$
then $W(\tilde{x})$ satisfies the quasilinear elliptic equation \eqref{self-w-eq} with $p=P$ in $D_0 := \{\tilde{x}\in \R^N \mid 0< |\tilde{x}| \leq 1\}$ and
$$
W(\tilde{x})  =1 \mbox{ on } \partial D_0,\quad 0\leq W(\tilde{x}) \leq 1 \mbox{ and } D_i W = \Phi'(z)  \frac{\tilde{x}_i}{|\tilde{x}|} \mbox { in } D_0.
$$
Therefore, both $W$ and $D_i W$ are bounded in $D_0$ by \eqref{1st-2nd-Phi}.  Using the standard $L^p$ theory for the elliptic equation
\eqref{self-w-eq} we see that $W\in W^2_q (D_0)$ for any $q>1$. Given $\mu \in (0,1)$,  when $q$ is large, $W^2_q (D_0)$ is embedded
into $C^{1+\mu} (\overline{D_0})$. Hence, by the Schauder theory we have $\|W\|_{C^{2+\mu} (\overline{D_0})} \leq C$.
Furthermore, by the standard regularity method we see that $W\in C^\infty (\overline{D_0})$, and for any positive integer $m$,
there exists $C(m)>0$ such that $\|W\|_{C^m (\overline{D_0}) } \leq C(m)$.  Consequently, we have
$$
\Phi \in C^\infty ([0,1]), \quad \|\Phi\|_{C^m ([0,1])} \leq C(m).
$$

(2). Next we show $\Phi'(0)=0$.  By contradiction we assume that for a sequence $\{z_j\}$ decreasing to $0$, $\Phi'(z_j) \geq \delta$ for some $\delta\in (0,k]$. Then using the equation of $\varphi$ we have
$$
\frac{\Phi''(z_j)}{1+ [\Phi'(z_j)]^2} = P(\Phi (z_j)-z_j \Phi' (z_j) )-\frac{N-1}{z_j}\Phi' (z_j) \leq P -\frac{N-1}{z_j} \delta <0,
$$
for sufficiently large $j$. This contradicts $\Phi''(z)\geq 0$ in \eqref{1st-2nd-Phi}. Hence $\Phi' (0)=0$ and so $\Phi$ is a smooth solution of \eqref{ODE}.

(3). We now prove the strict inequalities in \eqref{1st-2nd-Phi-strict}. Consider the equation for $W(\tilde{x})$ in the first step again.
Using the strong maximum principle we have $W(\tilde{x}) >0$ for all $\tilde{x}\in \overline{D_0}$. Thus $\Phi(0)>0$ and
$\Phi''(0)= P\Phi(0)>0$. Using a similar argument as in Lemma \ref{lem:varphi''>0} one can show that $\Phi''(z)>0$ in $[0,1]$, and so
\eqref{1st-2nd-Phi-strict} follows easily.

Finally, we prove the uniqueness. Assume by contradiction that the problem \eqref{ODE} with
$p=p_1\in [P_2, P_1]\backslash \{P\}$ also has a solution $\varphi(z; p_1)$.
Assume further that $p_1 >P$ (the case $p_1 <P$ is proved similarly). Taking $z=1$ in the
equations we find that $\varphi''(1;p_1) > \Phi''(1;k)$, and so
$\varphi'(z;p_1) < \Phi'(z;k)$ for $z$ satisfying $0<1-z \ll 1$ since $\varphi'(1;p_1) = \Phi'(1;k)=k$. Set
$$
\hat{z} := \max\{z \mid 0\leq z<1, \varphi'(z;p_1) = \Phi'(z;k)\}.
$$
(The set in the right hand side is non-empty since $0$ belongs to it.) Then
\begin{equation}\label{varphi-Phi}
\varphi(\hat{z};p_1) > \Phi(\hat{z};k),\quad \varphi' (\hat{z};p_1) =  \Phi' (\hat{z};k)
\mbox{\ \ and\ \ } \varphi'' (\hat{z};p_1) \leq  \Phi'' (\hat{z};k).
\end{equation}
Using the first inequality, the second equality, our assumption $p_1 >P$ and the fact
$\Phi(z;k)-z\Phi'(z;k)>0$ (by \eqref{1st-2nd-Phi-strict}) we conclude that
$$
p_1 [\varphi(\hat{z};p_1) - \hat{z} \varphi'(\hat{z};p_1)] > p_1 [\Phi(\hat{z};k)-\hat{z} \Phi'(\hat{z};k)]
> P[ \Phi(\hat{z};k)-\hat{z} \Phi'(\hat{z};k)].
$$
Substituting this result into the equations of $\varphi(z;p_1)$ and $\Phi(z;k)$ we have
$\varphi''(\hat{z};p_1) > \Phi''(\hat{z};k)$, contradicting the third inequality of \eqref{varphi-Phi}.
This proves the uniqueness.
\end{proof}

\begin{remark}\label{rem:existence-selfsimilar}
\rm As we can see from above that the proof for the existence of solutions to \eqref{ODE} is
complicated due to the presence of the term $(N-1)\varphi'(z)/z$. Our approach is
different from and much more complicated than the $N=1$ case as in \cite{CGK, GH, Koh} etc.
\end{remark}

\subsection{Comparison principle and a priori estimates for $u$ and $u_r$}
Since the graph of $u(x,t)$ lies in $\Omega$ and has different definition domains
for different $t$, it is convenient to introduce a new notation to compare them.

Assume, for $i=1$ and $2$, $w_i(r)$ are positive functions defined in $[0,\xi_i]$
with $w_i (\xi_i) =\xi_i$, then we write
$$
w_1 (r) \preceq  w_2 (r),
$$
if $\xi_1 \leq \xi_2$ and $w_1 (r)\leq w_2 (r)$ in $[0,\xi_1]$.

\begin{defn}
For $i=1,2$, let $u_i (r,t)$ be two functions defined in $\{(r,t)\mid 0\leq r \leq \xi_i(t),\ t>0\}$ such that
$u_i (\xi_i(t),t) =\xi_i(t)$. Then $u_1$ is called a lower solution of \eqref{p}-\eqref{boundary1} if
\begin{equation}\label{u1}
\left\{
 \begin{array}{l}
\displaystyle u_{1t} \leq {\frac { u_{1rr}}  { 1+ u^2_{1r}}} + \frac{(N-1)u_{1r}}{r}, \qquad  0< r< \xi_1(t),\ t>0,\\
u_{1r} (0, t) \geq  0,\quad u_{1r} (\xi_1(t),t) \leq k( u_1(\xi_1(t),t)) , \qquad t>0.
\end{array}
\right.
\end{equation}
$u_2$ is called an upper solution of \eqref{p}-\eqref{boundary1} if the opposite inequalities hold.
\end{defn}

\begin{lem}\label{comparison}
      Assume $u_1 (r,t)$ defined for $r\in [0,\xi_1(t)],\ t>0$ and $u_2(r,t)$ defined for $r\in [0,\xi_2(t)],\ t>0$ are lower and upper solutions of \eqref{p}-\eqref{boundary1}, respectively.
      If $u_1 (r,0) \preceq u_2 (r,0)$, then $u_1 (r, t) \preceq u_2 (r,t)$ for $t>0$.
\end{lem}

This lemma follows from the maximum principle directly. To avoid the unboundedness of $1/r$
in the last term of \eqref{p}, one can adopt the non-symmetric form \eqref{p01}-\eqref{p02} instead of
the symmetric form \eqref{p}-\eqref{boundary1} to use the maximum principle.

\medskip

Now we give  a priori  estimate for $u$.  For any $T>0$, denote $Q_T := \{(r,t)\; |\; 0 < r < \xi (t)\ {\rm
and}\ 0< t \leq T\}$. Let $k_0,\ k^0$ be the real numbers defined in \eqref{g-value}.  Denote
$$
\underline{u} (r,t) := \sqrt{2Pt}\; \Phi \left( {\frac {r}{\sqrt {2Pt}}}; k_0 \right)  ,\quad \bar{u} (r,t):= \sqrt{2P^0 t}\; \Phi
\left( {\frac {r}{\sqrt{2P^0 t}}}; k^0\right) ,
$$
where $(P, \Phi(z;k_0)) $ and $(P^0, \Phi(z; k^0))$ are the solutions of \eqref{ODE} with $k=k_0$ and $k=k^0$, respectively, as obtained in Proposition \ref{prop:exist-self-similar}.
Assume
\begin{equation}\label{initial-height}
\underline{u} (r,t_0) \preceq  u_0 (r) \preceq \bar{u} (r,t^0),
\end{equation}
then, by the comparison principle we have
\begin{equation}\label{u-bound}
\underline{u} (r,t+t_0) \preceq u(r,t; u_0) \preceq \bar{u} (r,t+t^0)
\end{equation}
provided the solution $u(r,t;u_0)$ of \eqref{p}-\eqref{boundary1}-\eqref{initial} exists.

Next, we give the following gradient estimates.

\begin{lem}\label{lem:gradient}
    Let $u(r,t)$ be the classical solution of \eqref{p}-\eqref{boundary1}-\eqref{initial} in $[0,T)$. Then
\begin{itemize}
\item[(i)] $u_r$ is bounded:
    $$
     |u_r (r,t)| \leq G:= \max\{\|u'_0\|_{C}, k^0\} <1,    \quad (r,t)\in  \overline{Q_T};
    $$

\item[(ii)] Assume further that $u'_0(r)\geq 0$ for $r\in [0,\xi(0)]$, then $u_r(r,t)>0$ in $Q_T$.
\end{itemize}
\end{lem}

\begin{proof}
From the problem of $u$ we obtain the problem for $\eta:= u_r$:
$$
\left\{
 \begin{array}{ll}
 \displaystyle \eta_t = \frac{\eta_{rr}}{1+\eta^2} - \frac{2\eta}{(1+\eta^2)^2} \eta_r^2
 +(N-1)\frac{\eta_r}{r} - (N-1)\frac{\eta}{r^2}, & 0<r<\xi(t), 0<t\leq T,\\
 \eta(0,t)=0,\ \eta(\xi(t),t)= k(u(\xi(t),t)), & 0<t\leq T,\\
 \eta(r,0)= u'_0(r), & 0\leq r\leq \xi(0).
 \end{array}
 \right.
$$
To exclude the singularity caused by $1/r$ and $1/r^2$,  one can first consider the problem in
smaller domains. More precisely, for any small $\delta \in (0,G)$, by $\eta(0,t)=0$ and by the continuity of $\eta$,
there exists a small $\varepsilon>0$ such that
$$
-\delta \leq \eta(\varepsilon, t) \leq \delta,\quad t\in [0,T].
$$
Then, using the maximum principle for $\eta$ in the domain $Q^\varepsilon_T :=\{(r,t) \mid \varepsilon<r<\xi(t), 0<t\leq T\}$
we conclude that $|\eta (r,t)| \leq G$ in $Q^\varepsilon_T$. Taking limit as $\varepsilon\to 0$ we obtain the first conclusion.
The second conclusion can be  proved similarly.
\end{proof}

\subsection{Convert the problem into a fixed domain}

Even with the a priori estimates obtained above, to study the local or global existence
of the solution to the problem \eqref{p}-\eqref{boundary1}-\eqref{initial}  by using the
standard theory of parabolic equations,  we still have two main difficulties: 
\begin{itemize}
\item[(1).] the spatial domain $[0, \xi(t)]$ changes over time; 
\item[(2).] there is some singularity in the last term
in the equation \eqref{p}. 
\end{itemize} 
To solve the first difficulty we can straighten the boundary in
several ways, such as, to use the isothermal coordinate as in \cite{LMN, MNL} or to use the
spherical coordinate, etc. To solve the second difficulty we adopt the equation \eqref{p01}
rather than \eqref{p}, though the calculation will become more complicated.
Hence in this subsection we will straighten the boundary of the spatial domain
and convert the equation \eqref{p01} into one in a fixed domain and without singularities.

We transfer the original domain
$\Omega:=\{(x,y)\in \R^{N+1}\mid y> |x|>0\}$ into a new one by using a new variable
$ \zeta = \frac{x}{y}$.
In the new coordinate system $(\zeta,y)$, $\Omega$ is expressed as a half cylinder
$D:= \{(\zeta,y)\mid 0 < \zeta <1,\  y >0\}$.
If $u(r,t)$ is a solution of \eqref{p} satisfying the estimates in the previous subsection,
then $\tilde{u}(x,t) := u(|x|, t)$ is a solution of \eqref{p01} satisfying $|D \tilde{u}| \leq G<1$.
For simplicity, we rewrite $\tilde{u}$ as $u$ again, and introduce a new unknown $ v(\zeta,t)$ as follows:
\begin{equation}\label{vu}
v(\zeta,t) = u (X(\zeta,t), t),
\end{equation}
with
\begin{equation}\label{def:x-y}
\zeta = Y (x,t) := {\frac {x} {  u(x,t)}} ,
\end{equation}
and $ x =X(\zeta,t)$ being the inverse function.  Then $|\zeta|\leq 1$ due to $(x, u(x,t))\in \Omega$.
By the implicit function theorem, the inverse function $x=X(\zeta,t)$ exists if
$$
{\frac {\partial (x- \zeta u(x,t))}{\partial x}}   =\det \big( I_{N\times N} - \zeta^{T}\cdot D u \big) \not= 0,
$$
where $I_{N\times N}$ is the $N$-th order unit matrix,  $\zeta^T$ is a column vector and $D u$ is a row one.
This is actually true by $|D u| \leq G<1$  and the following lemma.

\begin{lem}
Assume $a=(a_1,\cdots, a_N), \ b=(b_1, \cdots, b_N)\in \R^N$ satisfy $|a|, |b|<1$. Then $\det (I_{N\times N}
- a^T \cdot b) \not= 0$, where $a^T$ denotes the transposition of $a$.
\end{lem}

\begin{proof}
It is sufficient to show that $1$ is not an eigenvalue of the matrix $(a^T\cdot b)$. Assume by contradiction that
\begin{equation}\label{abc}
c^T = (a^T\cdot b) \cdot c^T = a^T\cdot (b\cdot c^T)
\end{equation}
for some $c \in \R^N \backslash \{0\}$. This equality implies that $\sigma:= b\cdot c^T \not = 0$.
Moreover, multiplying the equality \eqref{abc} by $b$ from left we obtain
$$
\sigma = b\cdot c^T = \sigma (b\cdot a^T).
$$
This contradicts the facts that $\sigma\not= 0$ and $|b\cdot a^T| \leq |a|\cdot |b|<1$.
\end{proof}

For simplicity, in the rest of this subsection, we write
$$
u_i := D_i u = \frac{\partial}{\partial x_i} u(x,t),\quad v_j := \frac{\partial}{\partial \zeta_j} v(\zeta,t), \quad i,j=1,2,\cdots,N,
$$
and write $\sum\limits_{i=1}^N v_i x_i = x\cdot Dv$ as $v_i x_i$.
Differentiating $u(x,t) = v(\zeta,t) $ in $x_i$ we have $u_i = v_k \frac{\partial \zeta_k}{\partial x_i}$.
Using \eqref{def:x-y} we have
$$
\frac{\partial \zeta_k}{\partial x_i} = \frac{\delta_{ki} u - x_k u_i}{u^2}, \quad i, k =1,2,\cdots, N,
$$
and so
\begin{equation}\label{u-i-1st}
(u^2 + x_k v_k ) u_i = u v_i,
\end{equation}
or, equivalently,
\begin{equation}\label{u-i-1st000}
(v + \zeta Dv ) D u = Dv ,\quad Du =\frac{Dv}{v +\zeta Dv}.
\end{equation}
Multiplying the first equality by $\zeta$ we have
\begin{equation}\label{zeta-Dv}
\zeta Dv  = \frac{ \zeta Du }{1- \zeta Du}  v
\end{equation}
with $|\zeta Du|\leq G$ due to $|\zeta|\leq 1$ and $ |Du|\leq G$.  So
$$
v+ \zeta Dv = \frac{v}{1-\zeta Du}\in \Big[ \frac{v}{1+G} ,\ \frac{v}{1- G} \Big].
$$
Moreover, by the second equality of \eqref{u-i-1st000} we have
\begin{equation}\label{est-of-Dv}
|Dv| = |Du| (v+ \zeta Dv) \leq \frac{v G }{1-G}.
\end{equation}
Differentiating \eqref{u-i-1st}  in  $x_j$ we have
$$
\Big( 2uu_j +\delta_{kj} v_k + x_k v_{kl} \frac{\partial \zeta_l}{\partial x_j} \Big) u_i
+ (u^2 + x_k v_k) u_{ij} = u_j v_i + u v_{il} \frac{\partial \zeta_l}{\partial x_j}.
$$
Denote $\Delta := v^2 + v (\zeta Dv) \in \Big[ \frac{v^2}{1+G} ,\ \frac{v^2}{1- G} \Big]$. By a direct but tedious calculation we obtain
\begin{equation}\label{u-2nd-express}
u_{ij} \Delta^3  =  v_{ij} \Delta^2 -  v v_i ( \zeta D v_j) \Delta - v v_j ( \zeta D v_i) \Delta - 2 v^3 v_i v_j + v^2 v_i v_j \zeta_k \zeta_l v_{kl}
\end{equation}
Finally, differentiating $u(x,t)=v(\zeta,t)$ with respect to $t$, in a similar way as above we have
$$
u_t = \frac{v^2 v_t} {\Delta}.
$$
Consequently, the equation \eqref{p01} is converted into
\begin{equation}\label{v-eq}
        v_t = a_{ij}(\zeta,v, Dv) v_{ij} + f(\zeta, v, Dv),\quad  \zeta\in B_1,\ t>0,
\end{equation}
where $B_1 := \{\zeta \in \R^N \mid |\zeta|<1\}$,
\begin{eqnarray*}
a_{ij}(\zeta, v, Dv ) & := &  \frac{1}{v^2} \Big( \delta_{ij} -  \frac{v^2 v_i v_j}{\Delta^2 + v^2 |Dv|^2} \Big) + \frac{1}{\Delta^2 } \Big( \delta_{mn} -  \frac{v^2 v_m v_n}{\Delta^2 + v^2 |Dv|^2} \Big) v_m v_n \zeta_i \zeta_j\\
& & - \frac{2}{v \Delta} \Big( \delta_{mj} -  \frac{v^2 v_m v_j}{\Delta^2 + v^2 |Dv|^2} \Big) v_m \zeta_i ,
\end{eqnarray*}
$$
f(\zeta,v, Dv ) = -\frac{2v}{\Delta^2} \Big( \delta_{ij} - \frac{v^2 v_i v_j}{\Delta^2  + v^2 |Dv|^2 } \Big) v_i v_j.
$$

Now we show that the equation \eqref{v-eq} is a uniform parabolic one in any finite time interval.
In fact, for any $\gamma=(\gamma_1,\cdots, \gamma_N)\in \R^N$ satisfying $|\gamma|=1$, denoting $q_1:= \gamma Dv,\ q_2 := \gamma \zeta $, we have
\begin{eqnarray*}
a_{ij}\gamma_i \gamma_j & = & \frac{1}{v^2} \Big( 1 -  \frac{v^2 q^2_1 }{\Delta^2 + v^2 |Dv|^2} \Big) +
\frac{1}{\Delta^2 } \Big( |Dv|^2 q^2_2 -  \frac{v^2 |Dv|^4 q_2^2 }{\Delta^2 + v^2 |Dv|^2} \Big)  - \frac{2}{v \Delta} \Big( q_1 q_2  -
\frac{v^2 |Dv|^2 q_1 q_2}{\Delta^2 + v^2 |Dv|^2} \Big) \\
 & \geq & \Big( \frac{1}{v} -  \frac{q_1 q_2 }{\Delta } \Big)^2 - \frac{1}{\Delta^2 + v^2 |Dv|^2}
 \Big( q_1 - \frac{v|Dv|^2 q_2}{\Delta} \Big)^2 = Q_+ \cdot Q_-,
\end{eqnarray*}
for
$$
Q_\pm := \Big( \frac{1}{v} -  \frac{q_1 q_2 }{\Delta } \Big) \pm  \frac{1}{\sqrt{\Delta^2 + v^2 |Dv|^2}}
\Big( q_1 - \frac{v|Dv|^2 q_2}{\Delta} \Big).
$$
Now we show that $Q_+$ and $Q_-$ are positive.
Since we are considering radially symmetric solutions $u(x,t)$, the converted unknown $v(\zeta,t)$
is also a radially symmetric one, so $Dv(\zeta,t)=0$ when $\zeta=0$, and $Dv$ is parallel to $\zeta$
when $\zeta\not= 0$. The former implies that $a_{ij}\gamma_i \gamma_j = 1/v^2 >0$ at $\zeta=0$.
The latter implies that $Dv = (\zeta Dv)  \frac{\zeta}{|\zeta|^2}$ when $\zeta\not=0$. Moreover,
under the additional condition $u'_0(r)\geq 0$ in Theorem \ref{thm:exist} we have by Lemma \ref{lem:gradient}
(ii) that $u_r>0$ and so $Du$ has the same direction as $x$. Hence $q:= \zeta Dv \geq 0$ and so
\begin{equation}\label{q12}
q_1 =\gamma Dv =  \frac{q q_2}{|\zeta|^2},\quad  |Dv|^2 q_2 = \frac{q^2}{|\zeta|^2 } q_2 = q q_1,\quad q_1 q_2 =
q\frac{q_2^2}{|\zeta|^2} \leq q.
\end{equation}
Therefore,
\begin{eqnarray*}
vQ_\pm & \geq & \Big( 1- \frac{vq}{\Delta } \Big) \pm \frac{v q_1}{\sqrt{\Delta^2 + v^2 |Dv|^2}}  \Big( 1- \frac{vq}{\Delta }\Big)\\
& = & \Big( 1 - \frac{q}{v+q} \Big) \Big( 1 \pm \frac{q_1}{\sqrt{ (v+q)^2 + |Dv|^2}}  \Big) >0.
\end{eqnarray*}
This proves $a_{ij}\gamma_i \gamma_j >0$, and so the equation is a uniform parabolic one.

On the other hand, the boundary condition \eqref{boundary1} implies that the contact angle $\phi$ between
the graph of $y=u(r,t)$ and the line $y=r$ satisfies
$$
\tan \Big( \frac{\pi}{4} - \phi\Big) = u_r (\xi(t),t) = k (u(\xi (t), t)).
$$
Hence the boundary condition \eqref{p02} corresponding to the equation \eqref{p01} is converted into
$$
\frac{\zeta Dv + (v+\zeta Dv) } {\sqrt{2 [ (v+\zeta Dv)^2 + |Dv|^2 ] }} = \cos \phi = \frac{1+k(v)}{\sqrt{2 (1+k^2(v))}},\quad
\zeta \in \partial B_1, \ t>0.
$$
Since $Dv=q \zeta$ with $q>0$ on $\partial B_1$, we have $q= \zeta Dv = |Dv|$, and so the boundary condition is simplified as
\begin{equation}\label{v-eq-bdry}
2 \zeta Dv + v = \frac{1+k(v)}{1-k(v)}, \quad \zeta \in \partial B_1, \ t>0.
\end{equation}

The problem \eqref{v-eq}-\eqref{v-eq-bdry} is a quasilinear parabolic equation with oblique boundary condition in the
fixed spatial domain $B_1$.
With the a priori estimate for $v$ (the same as that for $u$) and that for $Dv$ (i.e. \eqref{est-of-Dv}) in hand,
by using the standard theory for parabolic equations (see for example, \cite{Lie})
we have the following existence result.

\begin{lem}\label{globalv}
 Assume $u_0(r)$ is an admissible function with $u'_0(r)\geq 0$ and $v_0(\zeta)$ is its converted function as in \eqref{vu}.
 Then the problem \eqref{v-eq} with
 initial data $v(\zeta, 0)= v_0(\zeta)$  has a unique, radially symmetric time-global solution $v(\zeta,t)$. Moreover,
 for any $\mu \in (0,1)$ and $T>\delta>0$, with  $Q^{\zeta}_{T,\delta} := B_1 \times [\delta,T]$, we have
  $v\in C^{2+\mu, 1 +{\frac {\mu}{2}}}
 \Big(\overline{Q^{\zeta}_{T,\delta}} \Big)$ and
  $\| v(\zeta,t)\|_{C^{2+\mu, 1+{\frac {\mu}{2}} }  \big(\overline{Q^{\zeta}_{T,\delta}} \big)}$ $ \leq C$ for some positive
  $C$ depending on $T,\delta$ and $\mu$.
\end{lem}

\begin{remark}\label{rem:u''>0}\rm
We remark that the additional condition $u'_0(r)\geq 0$ in Theorem \ref{thm:exist} and Lemma \ref{lem:gradient}
leads to $u_r \geq 0$ and $q\geq 0$, which are only used to derive the uniform parabolicity of \eqref{v-eq}.
\end{remark}

\subsection{Global existence for the solution of \eqref{p}-\eqref{boundary1}-\eqref{initial}}

Let $v(\zeta,t)$ be the solution obtained in the previous lemma. Then, the formula \eqref{def:x-y}
$$
\zeta = \frac{x}{u(x,t)} = \frac{x}{v(\zeta,t)}
$$
defines an implicit function $\zeta = Y(x,t)$. Denote $\tilde{u}(x,t) := v(Y(x,t),t)$, then it is a solution of
\eqref{p01}-\eqref{p02}. Consequently, $u(r,t)= u (|x|,t) := \tilde{u}(x,t)$ is the solution of
\eqref{p}-\eqref{boundary1}-\eqref{initial}:

\begin{lem}\label{globalu}
    The problem \eqref{p}-\eqref{boundary1}-\eqref{initial} with $u'_0(r)\geq 0$ has a unique, time-global solution
    $u(r,t)$.  Moreover, for any $\mu\in (0,1)$ and $T>\delta>0$, with $Q_{T,\delta} := \{(r,t) \mid 0<r<\xi(t),\ \delta\leq t \leq T\}$,
    we have $u \in C^{2+\mu, 1 +{\frac {\mu}{2}}} (\overline{Q_{T,\delta}})$ and $\| u(x,t)\|_{C^{2+\mu, 1+{\frac {\mu}{2}} }
     (\overline{Q_{T,\delta}})} \leq C$ for some positive $C$ depending on $T,\delta$ and $\mu$.
\end{lem}

\section{Estimate by the Slowest Self-similar Solution}

In the previous section we use two self-similar  solutions to give the $L^\infty$ estimate \eqref{u-bound} for $u$, which means
that
\begin{eqnarray*}
0 & \leq & u(r,t)-\underline{u}(r,t+t_0) \leq \bar{u}(r,t+t^0)-\underline{u}(r,t+t_0)\\
& = & \sqrt{2P^0 (t+t^0)} \Phi\Big( \frac{r}{\sqrt{2P^0 (t+t^0)}} ; k^0\Big) - \sqrt{2P (t+t_0)}
\Phi\Big( \frac{r}{\sqrt{2P (t+t_0)}} ; k_0\Big).
\end{eqnarray*}
In particular, at $r=0$ we have, with $\Phi^0:= \Phi(0;k^0),\ \Phi:= \Phi(0;k_0)$,
$$
0  \leq  u(0,t)-\underline{u}(0,t+t_0) \leq \sqrt{2P^0 (t+t^0)} \Phi^0 - \sqrt{2P (t+t_0)}\Phi =O(1)\sqrt{t},
$$
as $t\to \infty$. This estimate is too rough since the propagation speed of $\underline{u}$ and $\bar{u}$ are completely different.
In this section, we will give a more precise estimate for the solution (as shown in Theorem \ref{thm:est})
by considering the homogenization limit
of the solution, that is, the case when the period $\varepsilon$ of the boundary function $k$ tends to $0$.
To do this, we will construct another {\it better} upper solution by using $\underline{u}$ rather than $\bar{u}$.

Recall that we assume \eqref{initial-est} in Theorem \ref{thm:est}, that is,
\begin{equation}\label{initial-compare}
\underline{u} (r, t_0) \preceq u_0(r) \preceq \underline{u} (r, s_0).
\end{equation}
Since $\underline{u} (r,t+ t_0)$ is a lower solution of the problem we have
\begin{equation}\label{underline-u<u}
\underline{u} (r,t+ t_0) \equiv  \sqrt{2P(t+ t_0)}\ \Phi\Big( \frac{r}{\sqrt{2P(t+ t_0)}};k_0 \Big) \preceq u(r,t),\quad t>0.
\end{equation}

 Now we construct a fine upper solution by using the same $\Phi(z;k_0)$ rather than $\Phi(z; k^0)$.
For simplicity, we write $\Phi (z; k_0)$ as $\Phi(z)$ in what follows.
Since $\Phi(1)=1$ and $\Phi'(1) =k_0 <1$, we can assume that
$\Phi(z) $ is defined in $[0,1+a]$ for some small $a>0$ with  $(1+a)^2 < \frac32$ and that
$$
\Phi'(z)\leq \bar{g}:= \frac{1+k_0}{2} <1,\quad  z\in [1,1+a].
$$
Then
$$
\Phi(z)\leq 1+b := \Phi(1+a) < 1+a,\quad  z\in [1,1+a].
$$
Denote
$$
R(t):= \sqrt{2Pt},\quad R_1(t):= \sqrt{2P(t+s_0)},\quad \widehat{R}_1(t) := (1+a)R_1(t),
$$
$$
U(r,t):= R(t) \Phi \Big(\frac{r}{R(t)} \Big),\quad  U_1 (r,t):= R_1 (t) \Phi \Big(\frac{r}{R_1 (t)} \Big).
$$
Then $U_1(r,t)$ is well-defined not only for $0\leq r \leq R_1(t)$ but also for $0\leq r \leq \widehat{R}_1 (t)$,
and
$$
U_1 (\widehat{R}_1   (t) ,t) = R_1(t) \Phi(1+a ) = (1+b) R_1(t),\quad t>0.
$$

\subsection{Estimate in the time interval $[0, O(\varepsilon^{-1/3}) ]$}
First, we prepare some notation. Set
$$
a_1 = b_1 := {\frac {1}{6}}, \quad \tau_1 := P  \varepsilon^{-2b_1},\quad l_1 :=  2 P \varepsilon^{-b_1} ,
$$
$$
\psi_1 (r,t) := L_1 \varepsilon^{1/2} \Big(N t+ {\frac {r^2}{2}} \Big),  \quad t\in [0,\tau_1],\ r\in [0, l_1],
$$
and
\begin{equation}\label{baru1}
u^+_1 (r,t) := U_1 (r, t)+  \psi_1(r,t),\quad t\in [0,\tau_1],\ r\in [0, l_1],
\end{equation}
where $L_1>0$ satisfies
\begin{equation}\label{def-L1}
 L_1 R_1^2 (0) > R_1(0) + 4 M_0 + 6 (N-1) k_0 (1+4k_0^2) \quad \mbox{with}\quad M_0:=  \max_{z\in [0,1+a]} \Phi''(z) .
\end{equation}
Then,
$$
0\leq \psi_1 (r,t)\leq M L_1 \varepsilon^{a_1},\quad r\in [0, l_1],\ t\in [0,\tau_1],
$$
where $M:= NP+2P^2$. When
$\varepsilon$ is sufficiently small (say, $\varepsilon \leq \epsilon_1^* := [P/(3s_0)]^3$),
$$
\widehat{R}_1  (t) = (1+a)\sqrt{2P(t+s_0)} \leq (1+a)\sqrt{2P(\tau_1 +s_0)} < 2 \sqrt{P\tau_1} =l_1,\quad t\in [0,\tau_1],
$$
and so
$$
U_1 (r, t ) \leq u^+_1(r,t) \leq U_1 (r,t) + M  L_1 \varepsilon^{a_1},\quad r\in [0, \widehat{R}_1  (t)],\ t\in [0,\tau_1],
$$
In particular, at $r= \widehat{R}_1 (t)$ we have
$$
u^+_1 ( \widehat{R}_1 (t),t) = R_1 (t) \Phi(1+a) +\psi_1 ( \widehat{R}_1  (t),t) \leq (1+b) R_1(t)  + M  L_1 \varepsilon^{a_1},
$$
and so
$$
\widehat{R}_1  (t) - u^+_1 ( \widehat{R}_1 (t),t) \geq  (a-b)R_1 (t) - M  L_1 \varepsilon^{a_1} \geq (a-b)R_1 (0) - M  L_1 \varepsilon^{a_1} >0,\quad t\in [0,\tau_1],
$$
provided $\varepsilon$ is sufficiently small (say, $\varepsilon \leq \epsilon_2^*:= [ (a-b)\sqrt{2Ps_0}/ (ML_1)]^6$). Therefore, the graph of $y = u^+_1(r,t)$
for $r\in [0, \widehat{R}_1  (t)]$ intersects the line $y=r$ at a point $(\eta_1(t), \eta_1 (t))$, and so
$$
\eta_1 (t) < \widehat{R}_1 (t),\quad t\in [0,\tau_1].
$$
(There is only one of such point since $ u^+_{1r} (r,t)<1$ in $r\in [0,\widehat{R}_1(t)]$ when $\ve$ is small).

Now we give the estimate for $u$ in the time interval $[0,\tau_1]$.

\begin{lem}\label{lem:tau-1}
The following estimates hold
\begin{equation}\label{est-11}
u(\cdot ,t) \preceq U_1 (\cdot ,t) + M  L_1 \varepsilon^{1/6}, \quad t\in [0,\tau_1],
\end{equation}
\begin{equation}\label{est-12}
u(\cdot ,\tau_1) \preceq U_1 (\cdot ,\tau_1+s_1),
\end{equation}
for some $s_1 = O(1)$.
\end{lem}

\begin{proof}
We prove the lemma by showing that $u^+_1$ is an upper solution in the time interval $[0,\tau_1]$.

First we show that
\begin{equation}\label{upperineq}
u^+_{1t} \geq  \frac{u^+_{1rr}}{1+ (u^+_{1r})^2} + \frac{N-1}{r} u^+_{1r},\quad r\in (0, \eta_1 (t)],\ t\in [0,\tau_1].
\end{equation}
In fact, for $r\in (0, \eta_1 (t)],\ t\in [0,\tau_1]$ we have $U_{1r} (r,t)>0,\  \psi_{1r} (r,t) > 0$
and $U_{1rr} (r,t)>0$, and so
\begin{eqnarray*}
& & u^+_{1t} - \frac{u^+_{1rr}}{1+(u^+_{1r})^2} - \frac{N-1}{r} u^+_{1r}\\
&= & \frac{U_{1rr}}{1+U^2_{1r}} + \frac{N-1}{r} U_{1r} + L_1 N \varepsilon^{1/2} - \frac{U_{1rr} +\psi_{1rr}}{1+(U_{1r} +\psi_{1r})^2} - \frac{N-1}{r} (U_{1r} + \psi_{1r}) \\
&= & \frac{U_{1rr}}{1+U^2_{1r}} - \frac{U_{1rr}}{1+(U_{1r} +\psi_{1r})^2}  + L_1 N \varepsilon^{1/2} - \frac{L_1 \varepsilon^{1/2}}{1+(U_{1r} +\psi_{1r})^2} - (N-1)L_1\varepsilon^{1/2} \\
& > &  \frac{U_{1rr}}{1+U^2_{1r}} -  \frac{U_{1rr} }{1+(U_{1r} +\psi_{1r})^2} =   \frac{U_{1rr} (2U_{1r} +\psi_{1r}) \psi_{1r}} {[1+U^2_{1r}]\cdot [1+(U_{1r} +\psi_{1r})^2]} \geq 0.
\end{eqnarray*}

Next we consider the boundary conditions. On the left boundary $r=0$, $u^+_1$
satisfies the homogeneous Neumann condition, the same as that for $u$. So the comparison
principle is applied on this boundary. On the right boundary $r= \eta_1(t)$, however,
the original boundary condition in \eqref{p} is an oblique one (or, a nonlinear Robin one).
Since it is nonlinear, the derivative of the solution on the boundary oscillates violently,
it is difficult to require that a constructed upper solution (like the above $u^+_1$) satisfies
such a boundary condition exactly. Therefore, we will compare the values of $u^+_1$ and $u$
instead of their derivatives on the right boundary, that is, we will show that
\begin{equation}\label{upperboundary}
\xi (t) \leq \eta_1 (t),\quad t\in [0,\tau_1],
\end{equation}
where $\xi(t)$ is the $r$-coordinate of the right end point of $u(r,t)$. For this purpose, we need the requirement
that the period $\ve$ of $g$ is sufficiently small.

By \eqref{initial-compare} we have
$$
\xi(0) \leq R_1 (0)= \sqrt{2Ps_0} < \eta_1(0).
$$
Therefore, \eqref{upperboundary} holds in the time interval $[0,s]$ when $s$ is small. Assume $[0,s]$ is the largest one of such intervals in $[0,\tau_1]$.
In what follows we  prove $s=\tau_1$ and so our lemma follows from the comparison principle.

Argue by contradiction, we  assume $0<s<\tau_1$. We will construct a short barrier just below the point $Q_1 :=
(\eta_1(s), \eta_1(s))$, which is a stationary solution and will block the real solution $u$ from
propagating over the barrier to reach the point $Q_1$, and so derive a contradiction.
More precisely, we construct the barrier from the point $Q^\star = (r^\star, r^\star)$, where
$r^\star \in [\eta_1(s) -\varepsilon, \eta_1(s))$ and $g(r^\star)=k_0$. (Such $r^\star$ exists since $g$ is $\varepsilon$-periodic).
By $u^+_1 (R_1 (s),s) > U_1 (R_1 (s),s)=R_1(s)$ we have  $R_1 (s)<\eta_1(s)$, and so
\begin{equation}\label{r-star>R(0)}
r^\star \geq  \eta_1(s) -\varepsilon > R_1 (s) - \varepsilon > R_1 (0)-\ve > \frac{R_1(0)}{2},
\end{equation}
when $\varepsilon$ is sufficiently small (say, $\ve <\epsilon_3^* := \sqrt{Ps_0/2}$).  The barrier is
the solution of the following initial value problem:
\begin{equation}\label{barrier-v}
\frac{v_{rr}}{1+v^2_r} + \frac{N-1}{r}v_r=0\ (r<r^\star),\quad v(r^\star )=r^\star,\quad v_r(r^\star )=k_0= g(v(r^\star)).
\end{equation}

First we prove that $v(r_\star)> u^+_1(r_\star, s)$ for  $r_\star := r^\star-\ve^{1/2}$.
It is easily seen that $v_{rr}<0$ in its existence interval. Assume $v_r(r_0) = 2k_0$  for some $r_0 <r^\star$.
Then, for $x\in I := [r_0, r^\star]$, $k_0 \leq v_r(r)\leq 2k_0$ and so $v_{rr} (r) \geq -K/r$ for
$K := 2(N-1)k_0 (1+4g^2_0)$. Thus,
\begin{equation}\label{v-r<g-0}
v_r (r) \leq k_0 + K [\ln r^\star  -\ln r],\quad r\in I.
\end{equation}
Since $k_0 + K [\ln r^\star  -\ln r] = 2k_0$ if and only if $r=r'_0 := r^\star e^{-k_0/K}$, we have $r_0 < r'_0$.
This implies that when $\varepsilon$ is sufficiently small (say, $\ve \leq \epsilon_4^* := [r^\star (1-e^{-k_0/K})]^2$),
we have  $r_\star = r^\star -\ve^{1/2} > r'_0 > r_0$, and so $r_\star \in I$ and
$$
\ln r^\star - \ln r_\star = - \ln \Big( 1- \frac{\ve^{1/2}}{r^\star}\Big)  < \frac{3 \ve^{1/2}}{2 r^\star}
< \frac{3 \ve^{1/2}}{R_1 (0)},
$$
the last inequality follows from \eqref{r-star>R(0)}. Therefore, by $v_{rr}<0$ and \eqref{v-r<g-0} we have
$$
v_r(r) \leq v_r (r_\star) \leq k_0 + K[\ln r^\star - \ln r_\star] \leq  k_0 + \frac{3K}{R_1 (0)} \varepsilon^{1/2},\quad
r\in [r_\star, r^\star]\subset I.
$$
Integrating this inequality over $[r_\star, r^\star]$ we have
\begin{equation}\label{v>=u}
v(r_\star) \geq  r^\star  -  k_0 \varepsilon^{1/2} - \frac{3K}{R_1 (0)} \varepsilon.
\end{equation}

On the other hand, for $r\in I_1 := [r_\star, \eta_1(s)]$ we have
$$
U_{1rr} (r,s) = \frac{\Phi''(\frac{r}{R_1 (s)})}{R_1 (s)} \leq M_2 := \frac{M_0}{R_1 (0)}.
$$
Hence, for any $r\in I_1 $, noting $R_1 (0)<R_1 (s)<\eta_1(s)$  we have
\begin{eqnarray*}
u^+_{1r} (r,s) & = & u^+_{1r} (\eta_1(s),s) +u^+_{1rr}(\theta,s) (r-\eta_1(s)),\quad \mbox{for some } \theta\in (r,\eta_1(s)),\\
 & \geq & U_{1r} (R_1 (s),s) +L_1 \varepsilon^{1/2} \eta_1(s) + (U_{1rr}(\theta,s)+L_1\varepsilon^{1/2}) (r-\eta_1(s))\\
& \geq & k_0 +L_1 \varepsilon^{1/2} R_1 (0)  - (M_2 +L_1 \varepsilon^{1/2}) (\eta_1(s)-r).
\end{eqnarray*}
Integrating it over $I_1$  we have
\begin{eqnarray*}
u^+_1 (r_\star, s) & \leq & u^+_1(\eta_1(s),s) - \big(k_0 +L_1 \varepsilon^{1/2} R_1 (0)\big) (\eta_1 (s)-r_\star) + (M_2+L_1\varepsilon^{1/2})\frac{(\eta_1(s)-r_\star)^2}{2}\\
& \leq & \eta_1(s) - \big(k_0 +L_1 \varepsilon^{1/2} R_1 (0)\big) (r^\star -r_\star) + 2 (M_2+L_1\varepsilon^{1/2})\ve \\
& < & r^\star +\varepsilon - k_0\varepsilon^{1/2} - L_1 R_1 (0)\varepsilon + 4 M_2 \ve  ,
\end{eqnarray*}
provided $\varepsilon$ is sufficiently small (say, $\ve \leq \epsilon_6^* := (M_2/ L_1)^2$, and so $L_1 \ve^{1/2} <M_2$).
Combining with \eqref{v>=u} we have
\begin{equation}\label{v>u+epsilon}
v(r_\star)-u^+_1(r_\star,s) \geq \Big( L_1 R_1 (0)-1 - 4 M_2 - \frac{3K}{R_1(0)} \Big)\varepsilon >0.
\end{equation}
The last inequality follow from the choice of $L_1$.

Now we prove that $v$ blocks the propagation of $u$ from it touching the point $Q_1$.
From above we see that $u^+_1$ is an upper solution in the time interval $[0,s]$ and so
$$
u(\cdot, t) \preceq  u^+_1(\cdot, t)\preceq u^+_1(\cdot, s),\quad t\in [0,s),
$$
where, the second inequality follows from the fact $u^+_{1t}>0$. Therefore, in the time interval $[0,s]$, the $r$-coordinate $\xi(t)$ of
the end point of $u(\cdot,t)$ either satisfies $\xi(t)<r_\star $, or $\xi(t)\geq r_\star $ but $u(r_\star,t) \leq u^+_1(r_\star ,t) \leq
u^+_1(r_\star ,s) < v(r_\star)$ by \eqref{v>u+epsilon}.
Since $v$ is a stationary solution to the equation and satisfies the boundary condition at the point $Q^\star$ (by \eqref{barrier-v}),
we see by the comparison principle that $u(\cdot,t) \preceq  v(\cdot)$ for any $t\in [0,s]$. This clearly contradicts
the assumption that $\xi(s)=\eta_1(s) > r^\star$. This contradiction then proves $s=\tau_1$.  Using comparison again in $[0,\tau_1]$ we see that $u^+_1$ is an upper solution and
$u\preceq u^+_1$ in $[0,\tau_1]$. This proves \eqref{est-11}.

Next, we prove \eqref{est-12}.  Taking $t= \tau_1$ in \eqref{est-11} we have
$$
u(\cdot ,\tau_1)\preceq U_1 (\cdot , \tau_1)  + M L_1 \varepsilon^{a_1} \preceq U_1 (\cdot ,\tau_1  +s_1),
$$
for some positive $s_1 $ which is taken as small as possible such that equality holds in the last inequality at some $r=\hat{r}_1$.
Then by $\Phi(z)-z\Phi'(z) \geq \delta_0$ in $z\in [0,1+a]$ for some $\delta_0 >0$, we have
$$
U_{1t} (r,t) = \frac{P}{R_1 (t)} \Big[  \Phi \Big(\frac{r}{R_1 (t)}\Big)  - \frac{r}{R_1(t)}  \Phi' \Big( \frac{r}{R_1(t)} \Big) \Big]
\geq  \frac{\delta}{2 \sqrt{t+s_0}} ,
$$
for $\delta := \sqrt{2P} \delta_0$, and
\begin{eqnarray*}
M L_1 \varepsilon^{a_1} & = & U_1 (\hat{r}_1,\tau_1 +s_1) - U_1 (\hat{r}_1,\tau_1)  =  \int_{\tau_1}^{\tau_1 +s_1} U_{1t} (\hat{r}_1, t)dt \\
& \geq & \delta \big( \sqrt{\tau_1 +s_1 +s_0} - \sqrt{\tau_1 +s_0} \big)  \\
& \geq &  \frac{\delta  s_1}{2\sqrt{\tau_1 + s_1 +s_0}}  \geq  \frac{\delta  s_1}{3\sqrt{\tau_1 + s_1}},
\end{eqnarray*}
provided $\ve$ is sufficiently small (say, $\ve \leq \epsilon_7^* := (P/s_0)^3$). So
$$
\delta^2 s_1^2 \leq 9 P M^2 L_1^2 + 9 M^2L_1^2 \ve^{1/3} s_1.
$$
This implies that $s_1 \leq 3ML_1 \sqrt{P} + O(\ve^{1/3}) = O(1)$, and so \eqref{est-12} is proved.
\end{proof}

\subsection{Estimate in the time interval $[O(\ve^{-1/3}), O(\ve^{-7/12})]$}
Denote
$$
\nu_2 := \varepsilon^{b_1} ,\quad \varepsilon_2 := \nu_2 \ve = \varepsilon^{1+b_1},\quad x_2:= \nu_2 x,\quad r_2 := \nu_2 r, \quad y_2  = \nu_2 y,\quad  t_2 := \nu^2_2 t,
$$
$$
\Omega_2 := \{(x_2, y_2) \in \R^{N+1} \mid y_2 >r_2 =|x_2|, x_2 \in \R^N\},
$$
and set
$$
u_2 (r_2 , t_2) := \nu_2  u\left( {\frac {r_2} {\nu_2  }}, {\frac
{t_2} {\nu^2_2 }} + \tau_1 \right)
$$
then the problem \eqref{p} is converted into the following problem
\begin{equation}\label{prob2}
\left\{
 \begin{array}{ll}
\displaystyle  u_{2t_2} = {\frac {u_{2 r_2 r_2}} {1+ u^2_{2 r_2}}} + \frac{N-1}{r_2} u_{2 r_2}, & 0 < r_2 < \xi_{2} (t_2) := \nu_2 \xi \big( \frac{t_2}{\nu_2^2} \big),\ t_2  > 0,\\
u_{2 r_2} (0,t_2)=0, & t_2>0,\\
\displaystyle u_{2 r_2} (r_2, t_2) =  k \Big( \frac{u_2 (r_2,t_2)}{\nu_2} \Big), & r_2 = \xi_{2} (t_2),\ t_2>0,
\end{array}
\right.
\end{equation}
Denote
$$
R_2 (t_2) := \nu_2 \sqrt{2P \Big( \frac{t_2}{\nu_2^2} +\tau_1 +s_1 +s_0\Big)} =  \sqrt{2P \big( t_2 + \nu_2^2 (\tau_1 +s_1 +s_0) \big)},
$$
and
$$
U_2 (r_2, t_2)  :=  \nu_2 U \left( \frac{r_2}{\nu_2}, \frac{t_2}{\nu_2^2} +\tau_1 +s_1 +s_0 \right) =    R_2(t_2) \ \Phi \left( {\frac {r_2} { R_2(t_2)}} \right),
$$
where $\tau_1, s_1$ are those given in \eqref{est-12}. Then it is easy to see from the previous lemma that
$$
u_2 (r_2,0) \preceq U_2 (r_2, 0).
$$

In the following, we repeat the argument in the previous subsection to give a fine upper estimate for $u$ in the time interval $[O(\ve^{-1/3}), O(\ve^{-7/12})]$. Set
$$
b_2 := {\frac {3}{28}},\quad a_2:= \frac12 - 2b_2=\frac27,\quad \tau_2 := P \ve_2^{-2b_2} ,\quad l_2 := 2P \ve_2^{-b_2} ,
$$
$$
\psi_2 (r_2,t_2) := L_2 \ve^{1/2}_2 \Big(N t_2+ {\frac {r_2^2}{2}} \Big), \quad t_2\in [0,\tau_2],\ r_2\in [0, l_2],
$$
and
\begin{equation}\label{baru1}
u^+_2 (r_2,t_2) := U_2 (r_2, t_2)+  \psi_2(r_2,t_2), \quad  t_2\in [0,\tau_2],\ r_2\in [0, l_2],
\end{equation}
where $L_2$ is a positive constant satisfying
$$
2P L_2 >1 ,\quad L_2 P^2 - P > 4M_0 + 6(N-1)k_0 (1+ 4k_0^2),
$$
and so $L_2 R_2^2(0) > R_2(0) +4M_0 + 6(N-1)k_0 (1+4k_0^2)$.  Note that
$$
0\leq \psi_2 (r_2,t_2)\leq M L_2 \ve_2^{a_2},\quad r_2\in [0, l_2],\ t_2\in [0,\tau_2],
$$
and so
$$
U_2 (r_2, t_2) \leq u^+_2 (r_2, t_2) \leq U_2 (r_2,t_2) + M  L_2 \ve_2^{a_2},\quad r_2\in [0, \widehat{R}_2  (t_2)],\ t_2\in [0,\tau_2],
$$
where
\begin{eqnarray*}
\widehat{R}_2  (t_2) & := & (1+a)R_2(t_2)\leq (1+a)R_2(\tau_2)  =  (1+a) \sqrt{2P [\tau_2 + \nu_2^2 (\tau_1 + s_1 +s_0)]} \\
& = &  (1+a) \sqrt{2P [P\ve_2^{-2b_2}  + \ve^{2b_1}  (\tau_1 + s_1 +s_0)]} < 2 P \ve_2^{-b_2} =l_2,\quad t\in [0,\tau_2].
\end{eqnarray*}
Hence,
$$
u^+_2 ( \widehat{R}_2  (t_2),t_2) = R_2 (t_2) \Phi (1+a) +\psi_2 ( \widehat{R}_2  (t_2), t_2) \leq (1+b) R_2(t_2)  + M  L_2 \ve_2^{a_2},
$$
and so
$$
\widehat{R}_2  (t_2) - u^+_2 ( \widehat{R}_2 (t_2),t_2) \geq  (a-b)R_2 (t_2) - M  L_2 \varepsilon^{a_2} > \sqrt{2} (a-b)P - ML_2 \ve^{a_2} >0,
$$
provided $\varepsilon$ is sufficiently small (say, $\ve \leq \epsilon_8^* := [\sqrt{2} (a-b) P/ (M L_2)]^{7/2}$).
Therefore, the graph of $u^+_2(r_2,t_2)$ intersects the line $y_2 = r_2$ at a point $(\eta_2(t_2), \eta_2(t_2))$ with
$$
\eta_2 (t_2) < \widehat{R}_2  (t_2),\quad t_2\in [0,\tau_2].
$$

\begin{lem}\label{lem:tau-2}
The following estimates hold
\begin{equation}\label{est-11-2}
u_2 (\cdot ,t_2) \preceq U_2 (\cdot ,t_2) + M  L_2 \ve_2^{a_2}, \quad t_2\in [0,\tau_2],
\end{equation}
\begin{equation}\label{est-12-2}
u_2 (\cdot ,\tau_2) \preceq U_2 (\cdot ,\tau_2+s_2),
\end{equation}
for some $s_2 = O(1)\ve^{5/24}$.
\end{lem}

\begin{proof}
The proof for \eqref{est-11-2} is similar as proving \eqref{est-11}. We now prove \eqref{est-12-2}. Taking $t_2 = \tau_2$ in \eqref{est-11-2} we have
$$
u_2 (\cdot  ,\tau_2)\preceq  U_2 (\cdot, \tau_2)  + M L_2 \ve_2^{a_2} \preceq U_2 (\cdot,\tau_2  +s_2),
$$
for some positive $s_2 $ which is taken as small as possible such that equality holds in the last inequality at some $r_2 =\hat{r}_2$. Then by
$$
U_{2t_2} (r_2,t_2) = \frac{P}{R_2 (t_2)} \Big[  \Phi \Big(\frac{r_2}{R_2(t_2)}\Big)  - \frac{r_2}{R_2(t)}
 \Phi' \Big( \frac{r_2}{R_2(t_2)} \Big) \Big] \geq  \frac{\delta }{2 \sqrt{t_2 + \nu_2^2 (\tau_1 +s_1 +s_0)}} ,
$$
we have
\begin{eqnarray*}
M L_2 \ve_2^{a_2} & = & U_2 (\hat{r}_2, \tau_2 +s_2) - U_2 (\hat{r}_2,\tau_2)  =  \int_{\tau_2}^{\tau_2 +s_2} U_{2t_2} (\hat{r}_2, t_2 )dt_2 \\
& \geq & \delta \Big( \sqrt{\tau_2 +s_2 + \nu_2^2 (\tau_1 +s_1+s_0)} - \sqrt{\tau_2 +\nu_2^2 (\tau_1 +s_1 +s_0)} \Big)\\
& \geq &  \frac{\delta s_2}{2\sqrt{\tau_2 + s_2 + \nu_2^2 (\tau_1 +s_1 +s_0)}} \geq
\frac{\delta s_2}{2\sqrt{\tau_2 + s_2 + 2P}}.
\end{eqnarray*}
This implies that
$$
\delta^2 s_2^2  \leq  4M^2 L_2^2 (s_2 +2P) \ve_2^{4/7} + 4M^2L_2^2 P \ve_2^{5/14}.
$$
Therefore $s_2 =  O(1) \ve_2^{5/28} = O(1) \ve^{5/24}$.
\end{proof}

\medskip
Now we use the original notation $r,t,u,U$ instead of $r_2,t_2, u_2, U_2$ to rewrite \eqref{est-11-2} we have
$$
u(\cdot, t+\tau_1 ) \preceq U(\cdot, t+\tau_1 +s_1 +s_0) + M  L_2 \ve^{1/6},\quad t\in [ 0, T_2],
$$
where $T_2 :=  \frac{\tau_1 \tau_2}{P} =  P \ve^{-2[(1+b_1)(1+b_2)-1]} $, or, equivalently,
\begin{equation}\label{est-2-stage}
u(\cdot, t) \preceq U(\cdot, t+s_1+s_0) + M  L_2 \ve^{1/6},\quad t\in [\tau_1, \tau_1 + T_2].
\end{equation}
Similarly, rewrite \eqref{est-12-2} by $r,t,u,U$ we have
\begin{equation}\label{est2end}
u(\cdot, T_2 +\tau_1) \preceq U(\cdot, T_2 + \tau_1 + S_2 +s_1 +s_0),
\end{equation}
with $S_2 := s_2 / \nu_2^2 = O(1)\ve^{-1/8}$.

\subsection{Induction}
Now we use induction.  Set
\begin{equation}\label{ab}
a_n := {\frac {5\cdot 4^{n-2} - 3^{n-1} }  {10\cdot 4^{n-2}
- 3^{n-1}}},\ \ \
b_n := {\frac { 3^{n-1} }  {10\cdot 4^{n-1} -4\cdot 3^{n-1}}}
\ \ \ \ \ {\rm for}\ n=1,2,3,\cdots,
\end{equation}
then $a_n + 2b_n = {\frac {1}{2}}$,
$$
a_1 = b_1 = {\frac {1}{6}},\ \ a_2 = {\frac {2}{7}},\ \
b_2 = {\frac {3}{28}},\ \ a_3 = {\frac {11}{31}},\ \
b_3 = {\frac {9}{124}}, \cdots.
$$
Furthermore, for $n=1,2,\cdots$, set
$$
\left\{
\begin{array}{l}
\displaystyle B_0 =1,\quad B_{n} :=  (1+b_1)(1+b_2)\cdots (1+b_n ) = {\frac {10\cdot 4^{n-1} - 3^{n}} {6\cdot 4^{n-1}}}\to \frac53\ (n\to \infty),\\
\nu_{n} := \varepsilon^{B_{n-1} -1},\quad \varepsilon_{n} := \nu_{n} \ve = \ve^{B_{n-1}},\quad \tau_{n}
:= P \varepsilon_{n}^{-2b_{n}} = P\ve^{-2b_{n} B_{n-1}},\\
T_{n} := P\ve^{2-2B_{n}}\to P\ve^{-\frac43}, \quad S_n:= \ve^{B_n (a_{n+1}-b_{n+1}) +2 - 2B_n} = \ve^{\frac12 [(\frac34)^{n-1} -1]} \to \ve^{-\frac12},\\
 \mathcal{T}_n := \tau_1 +T_2 +\cdots +T_n,\quad \mathcal{S}_n := s_0 + s_1 +S_2 +\cdots + S_n.
\end{array}
\right.
$$
Then
\begin{equation}\label{est-T_n-S_n}
\frac{\mathcal{T}_n}{n}\to P\ve^{-4/3},\quad \frac{\mathcal{S}_n}{n}\to \ve^{-1/2}\quad \mbox{as}\ n\to \infty.
\end{equation}
By induction, we suppose that, for $k=0,1,2,\cdots,n$, the following results hold
\begin{equation}\label{est-k-stage-1}
u(\cdot, t) \preceq U(\cdot, t+ \mathcal{S}_{k-1} ) + M  L_k  \ve^{1/6},\quad t\in [\mathcal{T}_{k-1}, \mathcal{T}_k],
\end{equation}
\begin{equation}\label{est-k-stage-2}
u(\cdot, \mathcal{T}_k) \preceq U (\cdot, \mathcal{T}_k + \mathcal{S}_k).
\end{equation}
Using a similar argument as above one can show that these inequalities also hold for $k=n+1$, and so hold for all $k\in \mathbb{N}$.

\medskip
\noindent
{\it Proof of Theorem} \ref{thm:est}. By \eqref{est-T_n-S_n} we have $\mathcal{T}_n \to \infty\ (n\to \infty)$.
Hence, any $t>0$ belongs to $[\mathcal{T}_{n-1}, \mathcal{T}_n)$ for some $n$.
Moreover, when $t$ is sufficiently large, it follows from \eqref{est-T_n-S_n} that
$$
\frac{P}{2}\ve^{-4/3} <\frac{\mathcal{T}_{n-1}}{n-1} \leq \frac{t}{n-1},\quad \frac{\mathcal{S}_{n-1}}{n-1} <2 \ve^{-1/2},
$$
and so
\begin{equation}\label{n-to-t}
\mathcal{T}_{n-1}\geq \frac{P(n-1)}{2}\ve^{-4/3},\quad n-1 \leq \frac{2}{P}\ve^{4/3}t,\quad \mathcal{S}_{n-1} \leq 2(n-1)\ve^{-1/2}.
\end{equation}

Denote $K:= \max\limits_{z\in [0,1+a]} \{\Phi(z) - z\Phi'(z)\}$, then
$$
U_t (r,t) = \frac{P}{R(t)}\Big[ \Phi\Big(\frac{r}{R(t)} \Big) - \frac{r}{R(t)} \Phi' \Big(\frac{r}{R(t)} \Big) \Big] \leq \frac{PK}{R(t)}.
$$
Combining with \eqref{est-k-stage-1} and \eqref{underline-u<u} we have
\begin{eqnarray*}
0 & \leq & u(r,t)- \underline{u}(r,t)\leq U(r,t+\mathcal{S}_{n-1}) + ML_n \ve^{1/6} - U(r,t+t_0)\\
& = & \int_{t+ t_0}^{t+\mathcal{S}_{n-1}} U_t (r,t) dt + ML_n \ve^{1/6}\\
& \leq & \sqrt{2P} K [\sqrt{t+\mathcal{S}_{n-1}} -\sqrt{t + t_0}] + ML_n \ve^{1/6}\\
& \leq & \frac{\sqrt{2P}K \mathcal{S}_{n-1}}{2\sqrt{t+t_0}} + ML_n \ve^{1/6}\\
& \leq & \frac{\sqrt{2P}K \mathcal{S}_{n-1}}{2\sqrt{\mathcal{T}_{n-1} }}  + ML_n \ve^{1/6}.
\end{eqnarray*}
By \eqref{n-to-t} we have
\begin{equation}\label{finer-est-but-with-t}
0 \leq  u(r,t)-\underline{u}(r,t) \leq 2K \sqrt{n-1} \ve^{1/6} +ML_n \ve^{1/6} = O(1) [\ve^{5/6}\sqrt{t} + \ve^{1/6}].
\end{equation}
When $t$ is not large we can take the coefficient sufficiently large such that the estimate holds for such $t$. Hence
this estimate can be true for all $t$.
In particular, when $t \in [0, O(1) \ve^{-4/3}]$ we have $\ve^{2/3}\sqrt{t} = O(1)$, and so
$$
0 \leq  u(r,t)-\underline{u}(r,t) \leq O(1) \ve^{1/6}.
$$
This proves Theorem \ref{thm:est}. 
\hfill $\Box$

\begin{remark}\label{rem:best-est}\rm
From \eqref{finer-est-but-with-t} we see that the error of $u-\underline{u}$ is $O(1)\ve^{1/6}$ in the time interval $[0, O(1)\ve^{-4/3}]$,
and it is $O(1)\ve^{5/6}\sqrt{t}$ beyond this interval. 
Though the former interval is very wide when $\ve\ll 1$, the latter error is not uniform in $t$.
It is natural to expect to give a uniform (independent of $t$) estimate in the whole time interval $[0,\infty)$. 
Since the derivative of the solution on the boundary oscillates violently, the approach for such uniform estimate 
must be highly non-trivial, as can be seen in this section. 
\end{remark}

\noindent
{\bf Acknowledgement.} The authors would like to thank Professors Lixia Yuan, Kelei Wang, Wei Zhao and Maolin Zhou
for their helpful discussion.


\end{document}